\documentclass[11pt]{article}
\usepackage{mathtools}
\usepackage{amssymb}
\usepackage{cleveref}
\usepackage{graphicx}
\usepackage{amsmath, amsthm}
\usepackage[margin=1.4in]{geometry}
\usepackage{tikz}
\usepackage{setspace}
\theoremstyle{definition}
\newtheorem{definition}{Definition}
\newtheorem{proposition}{Proposition}
\newtheorem{corollary}{Corollary}

\numberwithin{intassumption}{assumption}

\newtheorem{lemma}{Lemma}
\newtheorem{theorem}{Theorem}

\newcommand{\circlednum}[1]{ \tikz[baseline=(X.base)] \node(X) [draw, shape=circle, inner sep=0.5 pt] {\tiny #1};}

\begin{document}
	\title{A Note on the Solution of Circulant Real Linear Systems and its Sensitivity Analysis}
	\date{October 21, 2025}
	\author{Alessandro Guazzini\thanks{Department of Economics, University of Florence; Email: alessandro.guazzini@edu.unifi.it. This research did not receive any specific grant from funding agencies in the public, commercial, or not-for-profit sectors.}, Enrico Caricchio\thanks{Email: enrico.caricchio@edu.unifi.it}}
	\maketitle
	
\begin{abstract}
	Employing the Fast Fourier Transform we propose a ready-to-use solution to circulant real linear systems of equations, particularly useful when a broader theoretical analysis is involved. We also show that strict diagonal dominance of the matrix of coefficients is a sufficient condition for sign consistency between solutions and parameters in sensitivity analysis. 
\end{abstract}

\section{Introduction}
Several fields within physics, engineering, statistics and economics involve circulant linear systems of equations, their solutions and their sensitivity analysis. 

A considerable body of earlier literature has addressed this topic. See for instance, (Berg 1975) [1], (Chen 1987) [3], (Chao 1988) [2] and (Rojo 1990) [11]. In recent years, new methods have been presented by (El-Sayed 2005) [6], (Corless and Fillion 2013) [5], (Lin 2013) [9] and (Lin 2014) [10].

The above works have not reached a ready-to-use solution, convenient when a deeper theoretical analysis is entailed.

Utilizing the Fast Fourier Transform we fill this gap by providing an easily intelligible solution to circulant real linear systems, both for the case of a non-constant and a constant vector of known values. 

We also give strict diagonal dominance of the matrix of coefficients as a sufficient condition for sign consistency between solutions and parameters in sensitivity analysis.

An ideal application of the result is in the well-known (Salop 1979) [12] or (Chen and Riordan 2007) [4] economics models where the equilibrium allocations are found solving circulant real systems of linear equations, to which a straightforward solution and sensitivity analysis approach is of great benefit for their economic interpretations. 

\section{The Solution to a Circulant Real Linear System of Equations}
In this section we derive our proposed solution to a circulant real linear system. 

We first introduce the following.

\begin{definition} \label{def.mat}
A circulant real linear system is
\begin{equation*}
	Ax=b
\end{equation*}
where $A \in \mathbb{R}^{n\times n}$ is the non-singular matrix of coefficients such that the $(k,j)$ entry of $A$ is $a_{k,j}=a_{\left(j-k\right)\mod n}$. While $x\in\mathbb{R}^n$ is the vector of solutions and $b\in\mathbb{R}^n$ is the vector of known values such that the $j$ entry of $b$ is $b_j=f_j\left(b_{1j}\dots, b_{sj}\right)$ with $f_j:\mathbb{R}^{s}\rightarrow\mathbb{R}$ at least once continuously differentiable.
\end{definition}

\begin{proposition}\label{prop.gray}
	\begin{equation*}
	 A = F \Psi F^*
	\end{equation*} 
where $F\in\mathbb{C}^{n\times n}$ is the Fast Fourier Transform (FFT) matrix of $A$, or the matrix of $A$'s eigenvectors where the $j^\text{th}$ element of the $k^\text{th}$ eigenvector is the $j^\text{th}$ of $n$ distinct complex roots of unity $\frac{\omega_{jk}}{\sqrt{n}} = \frac{1}{\sqrt{n}}e^{\frac{-2\pi i jk}{n}}$, $F^{*}\in\mathbb{C}^{n\times n}$ is instead $A$'s FFT conjugate matrix and eventually $\Psi\in\mathbb{C}^{n\times n}$ is the diagonal matrix of $A$'s eigenvalues with the $k^\text{th}$ element being $\psi_k=\sum\limits_{j=0}^{n-1}a_{j} e^{\frac{-2\pi i jk}{n}}=\sum\limits_{j=0}^{n-1}a_j\left[\cos\left(\frac{2\pi j k}{n}\right)-i\sin \left(\frac{2 \pi j k}{n}\right)\right]= \sum \limits_{j=0}^{n-1}a_j \cos\left(\frac{2\pi j k}{n}\right) - i \sum \limits_{j=0}^{n-1}a_j \sin\left(\frac{2\pi j k}{n}\right)$.
\end{proposition} 
\begin{proof}
	See (Gray 2006, 32-34) [7].
\end{proof}

\begin{corollary} \label{cor.inv.mat}
	\begin{equation*}
		A^{-1}=F\Psi^{-1}F^*
	\end{equation*}
\end{corollary}
\begin{proof}
	See (Gray 2006, 35) [7].
\end{proof}

We now introduce our first result.

\begin{theorem}\label{th.1}
For any $l=0,\dots,n-1$ the $l^\text{th}$ element of the solution vector $x$ is 
	{\small\begin{align*}
x_{l} = & \frac{\sum\limits_{j=0}^{n-1}b_j}{n\sum\limits_{j=0}^{n-1}a_j} + 2\sum\limits_{k=1}^{\lfloor\frac{n-1}{2}\rfloor}\frac{\sum \limits_{j=0}^{n-1}\sum\limits_{m=0}^{n-1}a_j b_m \cos\left(\frac{2\pi k \left(j+m-l\right)}{n}\right)}{n\sum\limits_{j=0}^{n-1}\sum\limits_{m=0}^{n-1}a_j a_m \cos\left(\frac{2 \pi k \left(j-m\right)}{n}\right)}+
\begin{cases}
		\frac{\sum\limits_{j=0}^{n-1} \left(-1\right)^{j+l} b_j}{n \sum \limits_{j=0}^{n-1}\left(-1\right)^j a_j}& \text{ if n even}\\
		0 & \text{ if n odd}
\end{cases}
    \end{align*}}
\end{theorem}
\begin{proof}
	We first note that $x=A^{-1}b=F\Psi^{-1}F^*b$ by \cref{def.mat} and \cref{cor.inv.mat}.
	
	Then, using \cref{prop.gray} and \cref{cor.inv.mat} we can readily observe that for any $l=0,\dots,n-1$
	\begin{equation*}
		x_l=\frac{1}{n}\sum\limits_{j=0}^{n-1}b_j\sum\limits_{k=0}^{n-1}\omega_{lk}\omega_{kj}^{*}\psi_{k}^{-1}=\sum\limits_{k=0}^{n-1}\psi_{k}^{-1}\frac{1}{n}\sum\limits_{j=0}^{n-1}b_{j}e^{\frac{2\pi i kj}{n}}e^{\frac{-2\pi i lk}{n}}
	\end{equation*}
	or 
	\begin{equation*}
		x_l=\sum\limits_{k=0}^{n-1}\psi_{k}^{-1}T_{k}e^{\frac{-2\pi i lk}{n}}
	\end{equation*}
	with $T_k=\frac{1}{n}\sum\limits_{j=0}^{n-1}b_{j}e^{\frac{2\pi i kj}{n}}$.

Furthermore, noting that 
\begin{equation*}
	T_{n-k} = \frac{1}{n}\sum\limits_{j=0}^{n-1}b_{j}e^{\frac{2\pi i (n-k)j}{n}} =\frac{1}{n}\sum\limits_{j=0}^{n-1}b_{j}e^{\frac{2\pi i nj}{n}}e^{\frac{-2\pi i kj}{n}} = \frac{1}{n} \sum_{j=0}^{n-1} b_j (-1)^{2j} e^{\frac{-2\pi i kj}{n}} = \overline{T_{k}}
\end{equation*} 
and similarly that
\begin{equation*}
	\psi_{n-k} =  \sum\limits_{j=0}^{n-1} a_j e^{\frac{-2\pi i j (n-k)}{n}}=  \sum\limits_{j=0}^{n-1} a_j e^{-2\pi i j} e^{\frac{2 \pi i j k}{n}} =  \sum\limits_{j=0}^{n-1} a_j \left(e^{\pi i}\right)^{-2j} e^{\frac{2 \pi i j k}{n}} =\overline{\psi_{k}}
\end{equation*}
it follows that \begin{align*}
&\psi_k^{-1}T_ke^{\frac{-2\pi i l k}{n}}+\psi_{n-k}^{-1}T_{n-k}e^{\frac{-2 \pi i l \left(n-k\right)}{n}}=\psi_k^{-1}T_{k}e^{\frac{-2\pi i lk}{n}} + \overline{\psi_{k}^{-1}}\overline{T_{k}}e^{\frac{-2\pi i l n}{n}}e^{\frac{2\pi i lk}{n}} = \\& \psi_{k}^{-1}T_{k}e^{\frac{-2\pi i lk}{n}} +\overline{\psi_{k}^{-1}A_{k}e^{\frac{-2\pi i lk}{n}}} = 2\mathfrak{R}\left(\psi_{k}^{-1}T_{k}e^{\frac{-2\pi i lk}{n}}\right)
\end{align*}

Eventually, since there are $\lfloor\frac{n-1}{2} \rfloor$ unique couples $\left(k,n-k\right)$ with a distinguished $k=\frac{n}{2}$ if $n$ is even, $x_l$ might be rewritten as 
\begin{equation*}
		x_{l} = \psi_0^{-1}T_0 + \sum\limits_{k=0}^{\lfloor\frac{n-1}{2}\rfloor}2\mathfrak{R}\left(\psi_{k}^{-1}T_{k}e^{\frac{-2\pi i l k}{n}}\right)+\begin{cases}
			\left(-1\right)^{l}\psi_{\frac{n}{2}}^{-1}T_{\frac{n}{2}} & \text{ if n even}\\
			0 & \text{ if n odd}
		\end{cases}
\end{equation*}

The proof is completed noting that \begin{align*}
	&\psi_0^{-1}=\left(\sum\limits_{j=0}^{n-1}a_je^{\frac{-2\pi ij 0}{n}}\right)^{-1}=\frac{1}{\sum\limits_{j=0}^{n-1}a_j}, \\
	&T_0=\frac{1}{n}\sum\limits_{j=0}^{n-1}b_j, \\
	&\psi_{\frac{n}{2}}^{-1}=\left(\sum\limits_{j=0}^{n-1}a_je^{\frac{-2 \pi i j \frac{n}{2}}{n}}\right)^{-1}=\left(\sum\limits_{j=0}^{n-1}a_j\left(e^{\pi i}\right)^{-j}\right)^{-1}=\frac{1}{\sum\limits_{j=0}^{n-1} \left(-1\right)^{j}a_j}, \\
	&T_{\frac{n}{2}}=\frac{1}{n}\sum\limits_{j=0}^{n-1}b_j e^{\frac{2\pi i \frac{n}{2} j}{n}}=\frac{1}{n}\sum\limits_{j=0}^{n-1}\left(-1\right)^{j} b_j
	\end{align*} 
	and that
	\begin{equation}
	\begin{aligned}
	&\mathfrak{R}\left(\psi_k^{-1}T_k e^{\frac{-2 \pi i l k}{n}}\right)=\mathfrak{R}\left(\psi_k^{-1}\right)\mathfrak{R}\left(T_k e^{\frac{-2 \pi i l k}{n}}\right)-\mathfrak{I}\left(\psi_k^{-1}\right)\mathfrak{I}\left(T_k e^{\frac{-2 \pi i l k}{n}}\right)\overset{\circlednum{1}}{=} \\
	& \frac{\sum \limits_{j=0}^{n-1} a_j \cos\left(\frac{2\pi j k}{n}\right)\sum \limits_{j=0}^{n-1}b_j \cos \left(\frac{2 \pi k \left(j-l\right)}{n}\right)-\sum \limits_{j=0}^{n-1} a_j \sin\left(\frac{2\pi j k}{n}\right)\sum \limits_{j=0}^{n-1}b_j \sin \left(\frac{2 \pi k \left(j-l\right)}{n}\right)}{n\left[\left[\sum \limits_{j=0}^{n-1}a_j \cos\left(\frac{2\pi j k}{n}\right)\right]^2+\left[\sum \limits_{j=0}^{n-1}a_j \sin\left(\frac{2\pi j k}{n}\right)\right]^2\right]}=\\
	& \frac{\sum \limits_{j=0}^{n-1}\sum\limits_{m=0}^{n-1}a_j b_m \cos\left(\frac{2\pi k \left(j+m-l\right)}{n}\right)}{n\sum\limits_{j=0}^{n-1}\sum\limits_{m=0}^{n-1}a_j a_m \cos\left(\frac{2 \pi k \left(j-m\right)}{n}\right)}
\end{aligned}
\end{equation}
with $\circlednum{1}$ following by $\psi_k\psi_k^{-1}=$\\ $\left[\sum \limits_{j=0}^{n-1}a_j \cos\left(\frac{2\pi j k}{n}\right) - i \sum \limits_{j=0}^{n-1}a_j \sin\left(\frac{2\pi j k}{n}\right)\right] \frac{\sum \limits_{j=0}^{n-1}a_j \cos\left(\frac{2\pi j k}{n}\right) + i \sum \limits_{j=0}^{n-1}a_j \sin\left(\frac{2\pi j k}{n}\right)}{\left[\sum \limits_{j=0}^{n-1}a_j \cos\left(\frac{2\pi j k}{n}\right)\right]^2+\left[\sum \limits_{j=0}^{n-1}a_j \sin\left(\frac{2\pi j k}{n}\right)\right]^2}=1$.
\end{proof}

A special case of \cref{th.1} is also presented.

\begin{proposition} \label{prop.cons}
	If the vector of known values $b$ is such that for any $j$ $b_j=f\left(b_1,\dots,b_s\right)=\beta$, the $l^\text{th}$ element of the solution vector $x$ is
	\begin{equation*}
		x_l=\beta\left(\sum_{j=0}^{n-1}a_j\right)^{-1}
	\end{equation*}
\end{proposition}
\begin{proof}
	Noting again that $x=A^{-1}b=F\Psi^{-1}F^*b$ by \cref{def.mat} and \cref{cor.inv.mat}, it is easily verified that 
	\begin{equation*}
		x_l=\frac{1}{n}\sum\limits_{j=0}^{n-1}\beta\sum\limits_{k=0}^{n-1}\omega_{lk}\omega_{kj}^{*}\psi_{k}^{-1}=\frac{\beta}{n}\sum_{k=0}^{n-1}e^{\frac{-2 \pi i l k}{n}}\psi_k^{-1}\sum_{j=0}^{n-1}e^{\frac{2 \pi i k j}{n}}
	\end{equation*}
	by proposition 1.
	
	Then, by the orthogonality of the complex exponential
	\begin{equation*}
		\sum_{j=0}^{n-1}e^{\frac{2 \pi i k j}{n}}=\begin{cases}
			n & \text{if $k \mod n = 0$} \\ 0 & \text{otherwise}
		\end{cases}
	\end{equation*}
	since for any $k$, $k<n$ it follows that
	\begin{equation*}
		x_l=\beta\psi_0^{-1}
	\end{equation*}
	
	Eventually, for $\psi_0=\sum\limits_{j=0}^{n-1}a_je^{\frac{-2\pi i j 0}{n}}=\sum\limits_{j=0}^{n-1}a_j$ the proof is completed.
\end{proof}

\section{A Sufficient Condition for Sign Consistency in Sensitivity Analysis}
We finally introduce strict diagonal dominance of the matrix of coefficients as a sufficient condition for sign consistency between solutions and parameters in sensitivity analysis.

\begin{lemma}\label{lem}
	If $A$ is such that $a_0>0$ and $a_0>\sum\limits_{j=1}^{n-1} |a_j|$, i.e. $A$ is \emph{strictly diagonally dominant}, then for any $k$, $\mathfrak{R}\left(\psi_k\right)>0$.
\end{lemma}
\begin{proof}
	See (Horn and Johnson 1990, 349).
\end{proof}

Now, recalling that by definition 1 $b_l=f_l\left(b_{1l},\dots,b_{rl}, \dots, b_{sl}\right)$, our second result follows.

\begin{theorem}\label{th.2}
	For any $l=0,\dots,n-1$ and for any $r=1,\dots, s$, if $A$ is strictly diagonally dominant, then 
	\begin{equation*}
		\frac{\partial x_l}{\partial b_{rl}}\ge0 \ \text{if and only if} \ \frac{\partial f_l}{\partial b_{rl}} \ge 0
	\end{equation*}
\end{theorem}
\begin{proof}
	For any $l=0,\dots,n-1$ and for any $r=1,\dots,s$, the partial derivative of $x_l$ with respect to $b_{rl}$ might be written as
	{ \begin{align*}
		\frac{\partial x_l}{\partial b_{rl}}=&\frac{\frac{\partial f_l}{\partial b_{rl}}}{n\sum\limits_{j=0}^{n-1}a_j}+\frac{2}{n}\frac{\partial f_l}{\partial b_{rl}}\sum\limits_{k=0}^{\lfloor\frac{n-1}{2}\rfloor}\frac{\sum \limits_{j=0}^{n-1}a_j \cos\left(\frac{2\pi j k}{n}\right)}{\left[\sum \limits_{j=0}^{n-1}a_j \cos\left(\frac{2\pi j k}{n}\right)\right]^2+\left[\sum \limits_{j=0}^{n-1}a_j \sin\left(\frac{2\pi j k}{n}\right)\right]^2}+ \\
		&\begin{cases}
			\frac{\sum\limits_{j=0}^{n-1} \left(-1\right)^{2l}\frac{\partial f_l}{\partial b_{rl}}}{n \sum \limits_{j=0}^{n-1}\left(-1\right)^j a_j}& \text{ if n even}\\
			0 & \text{ if n odd}
		\end{cases}
	\end{align*}}

   We note that if $A$ is strictly diagonally dominant, then $a_0+\sum \limits_{j=1}^{n-1}a_j\ge a_0-\sum\limits_{j=1}^{n-1} |a_j|\overset{\circlednum{2}}{>}0$ with $\circlednum{2}$ following by \cref{lem}.
   
   Similarly, $a_0+\sum \limits_{j=1}^{n-1}\left(-1\right)^{j} a_j\overset{\circlednum{3}}{\ge} a_0-\sum\limits_{j=1}^{n-1} |a_j|>0$ with $\circlednum{3}$ being true for $a_0+\sum\limits_{j \ \mathrm{even}} a_j + \sum\limits_{k \ \mathrm{odd}} -a_k \ge a_0-\sum\limits_{j \ \mathrm{even}} |a_j| - \sum\limits_{k \ \mathrm{odd}} |a_k|$.
   
   The proof is eventually completed by \cref{lem} noting that \begin{equation*} \frac{\sum \limits_{j=0}^{n-1}a_j \cos\left(\frac{2\pi j k}{n}\right)}{\left[\sum \limits_{j=0}^{n-1}a_j \cos\left(\frac{2\pi j k}{n}\right)\right]^2+\left[\sum \limits_{j=0}^{n-1}a_j \sin\left(\frac{2\pi j k}{n}\right)\right]^2}=\mathfrak{R}\left(\psi_k^{-1}\right)\end{equation*}
\end{proof}

	\section*{Bibliography}
	\begin{enumerate}
		\item Berg, L. (1975), Solution of Large Linear Systems with Help of Circulant Matrices. Z. angew. Math. Mech., 55: 439-441. \\ https://doi.org/10.1002/zamm.19750550714
		\item Chao, C.Y. (1988), A remark on symmetric circulant matrices. Linear Algebra and its Applications, 103: 133-148. \\
		https://doi.org/10.1016/0024-3795(88)90225-X.
		\item Chen, M. (1987). On the Solution of Circulant Linear Systems. SIAM Journal on Numerical Analysis, 24(3), 668–683. \\ http://www.jstor.org/stable/2157355
		\item Chen, Y. and Riordan, M. H. (2007), Price and Variety in the Spokes Model*. The Economic Journal, 117: 897-921. \\ https://doi.org/10.1111/j.1468-0297.2007.02063.x
		\item Corless, R. M. and Fillion, Nicolas (2013), A Graduate Introduction to Numerical Methods. From the Viewpoint of Backward Error Analysis. Springer New York, NY \\ https://doi.org/10.1007/978-1-4614-8453-0
		\item El-Sayed, S. M. (2005), A direct method for solving circulant tridiagonal block systems of linear equations. Applied Mathematics and Computation, 165(1): 23-30. \\ https://doi.org/10.1016/j.amc.2004.06.041.
		\item Gray, R. M. (2006), Toeplitz and Circulant Matrices: A Review. Foundations and Trends in Communications and Information Theory: Vol. 2(3): 155-239. \\ http://dx.doi.org/10.1561/0100000006
		\item Horn, R. A., and Johnson C. R. (1990), Matrix Analysis. 1 paperback ed. New York: Cambridge University Press. 
		\item Lin, F. (2013), The solution of linear systems equations with circulant-like coefficient matrices. Applied Mathematics and Computation, 219(15): 8259-8268. \\
		https://doi.org/10.1016/j.amc.2013.02.021.
		\item Lin, F. (2014), A direct method for solving block circulant banded system of linear equations. Applied Mathematics and Computation, 232: 1269-1276. \\
		https://doi.org/10.1016/j.amc.2014.01.099.
		\item Rojo, O. (1990), A new method for solving symmetric circulant tridiagonal systems of linear equations. Computers \& Mathematics with Applications, 20(12): 61-67. \\ https://doi.org/10.1016/0898-1221(90)90165-G.
		\item Salop, S. C. (1979). Monopolistic Competition with Outside Goods. The Bell Journal of Economics, 10(1): 141–156. \\ https://doi.org/10.2307/3003323
	\end{enumerate}
\end{document}